\documentclass[12pt]{article}
\usepackage{amssymb,amsthm,amsmath,amsfonts,latexsym,tikz,hyperref}
\usepackage[hmargin=1in,vmargin=1in]{geometry}

\newcommand{\ben}{\begin{enumerate}}
\newcommand{\een}{\end{enumerate}}
\newcommand{\ble}{\begin{lem}}
\newcommand{\ele}{\end{lem}}
\newcommand{\bth}{\begin{thm}}
\renewcommand{\eth}{\end{thm}}
\newcommand{\bpr}{\begin{prop}}
\newcommand{\epr}{\end{prop}}
\newcommand{\bco}{\begin{cor}}
\newcommand{\eco}{\end{cor}}
\newcommand{\bcon}{\begin{conj}}
\newcommand{\econ}{\end{conj}}
\newcommand{\bde}{\begin{defn}}
\newcommand{\ede}{\end{defn}}
\newcommand{\bex}{\begin{exa}}
\newcommand{\eex}{\end{exa}}
\newcommand{\barr}{\begin{array}}
\newcommand{\earr}{\end{array}}
\newcommand{\btab}{\begin{tabular}}
\newcommand{\etab}{\end{tabular}}
\newcommand{\beq}{\begin{equation}}
\newcommand{\eeq}{\end{equation}}
\newcommand{\bea}{\begin{eqnarray*}}
\newcommand{\eea}{\end{eqnarray*}}
\newcommand{\bal}{\begin{align*}}
\newcommand{\bce}{\begin{center}}
\newcommand{\ece}{\end{center}}
\newcommand{\bpi}{\begin{picture}}
\newcommand{\epi}{\end{picture}}
\newcommand{\bpp}{\begin{picture}}
\newcommand{\epp}{\end{picture}}
\newcommand{\bfi}{\begin{figure} \begin{center}}
\newcommand{\efi}{\end{center} \end{figure}}
\newcommand{\bprf}{\begin{proof}}
\newcommand{\eprf}{\end{proof}\medskip}

\newcommand{\bsl}{\begin{slide}{}}
\newcommand{\esl}{\end{slide}}
\newcommand{\bfr}{\begin{frame}}
\newcommand{\efr}{\end{frame}}

\newcommand{\hqed}{\hfill \qed}

\newcommand{\eqed}[1]{$\textcolor{white}{\qed}\hfill{\dil#1}\hfill\qed$}

\newcommand{\hso}[1]{\hspace{-1pt}}

\newcommand{\sbe}{\subseteq}

\newcommand{\ptn}{\vdash}

\newcommand{\fl}[1]{\lfloor #1 \rfloor}
\newcommand{\ce}[1]{\lceil #1 \rceil}

\def\<{\langle}
\def\>{\rangle}

\newcommand{\ree}[1]{(\ref{#1})}

\newcommand{\ra}{\rightarrow}

\newcommand{\la}{\lambda}

\newcommand{\si}{\sigma}

\newcommand{\fS}{{\mathfrak S}}

\DeclareMathOperator{\inv}{inv}

\DeclareMathOperator{\maj}{maj}

\newcommand{\dil}{\displaystyle}

\newtheorem{thm}{Theorem}[section]
\newtheorem{prop}[thm]{Proposition}
\newtheorem{cor}[thm]{Corollary}
\newtheorem{lem}[thm]{Lemma}
\newtheorem{conj}[thm]{Conjecture}
\newtheorem{exa}[thm]{Example}

\DeclareMathOperator{\lb}{lb}
\DeclareMathOperator{\ls}{ls}
\DeclareMathOperator{\rb}{rb}
\DeclareMathOperator{\rs}{rs}
\DeclareMathOperator{\LB}{LB}
\DeclareMathOperator{\LS}{LS}
\DeclareMathOperator{\RB}{RB}
\DeclareMathOperator{\RS}{RS}

\begin{document}
\pagestyle{plain}

\title{Set partition patterns and statistics
}
\author{Samantha Dahlberg\\[-5pt]
\small Department of Mathematics, Michigan State University,\\[-5pt]
\small East Lansing, MI 48824-1027, USA, {\tt dahlbe14@msu.edu}\\
Robert Dorward\\[-5pt]
\small Department of Mathematics, Oberlin College,\\[-5pt]
\small Oberlin, OH 44074, USA, {\tt rdorward@oberlin.edu}\\
Jonathan Gerhard\\[-5pt]
\small Department of Mathematics, James Madison University\\[-5pt]
\small Harrisonburg, VA 22801, USA, {\tt gerha2jm@dukes.jmu.edu} \\
Thomas Grubb\\[-5pt]
\small Department of Mathematics, Michigan State University,\\[-5pt]
\small East Lansing, MI 48824-1027, USA, {\tt grubbtho@msu.edu}\\
Carlin Purcell\\[-5pt]
\small Vassar College, 124 Raymond Ave. Box 1114, \\[-5pt]
\small Poughkeepsie, NY 12604 USA, {\tt carlinpurcell@gmail.com}\\
Lindsey Reppuhn\\[-5pt]
\small Department of Mathematics, Kalamazoo College\\[-5pt]
\small Kalamazoo, MI 49006, USA, {\tt Lindsey.Reppuhn11@gmail.com} \\
Bruce E. Sagan\\[-5pt]
\small Department of Mathematics, Michigan State University,\\[-5pt]
\small East Lansing, MI 48824-1027, USA, {\tt sagan@math.msu.edu}
}

\date{\today\\[10pt]
	\begin{flushleft}
	                                       \\[5pt]
	\end{flushleft}}

\maketitle

\begin{abstract}
A set partition $\si$ of $[n]=\{1,\dots,n\}$ contains another set partition $\pi$ if restricting $\si$ to some $S\sbe[n]$ and then standardizing the result gives $\pi$.  Otherwise we say $\si$ avoids $\pi$.   For all sets of patterns consisting of partitions of  $[3]$, the sizes of the avoidance classes were determined by Sagan and by Goyt.   Set partitions are in bijection with restricted growth functions (RGFs) for which Wachs and White defined four fundamental statistics.  We consider the distributions of these statistics over various avoidance classes, thus obtaining multivariate analogues of the previously cited cardinality results.  This is the first in-depth study of such distributions.  We end with a list of open problems.  

\end{abstract}

%
%

\section{Introduction}

There has been an explosion of papers recently dealing with pattern containment and avoidance in various combinatorial structures.  And the study of statistics on combinatorial objects has a long and venerable history.  By comparison, there are relatively few papers which study a variety of statistics on a number of different avoidance classes.   The focus of the present work is pattern avoidance in set partitions combined with four important statistics defined by Wachs and White~\cite{ww:pqs}.  It is the first comprehensive study of these statistics on avoidance classes.  In particular, we consider the distribution of these statistics over every class avoiding a set of partitions of $\{1,2,3\}$.   
We will start by providing the necessary definitions and setting notation.

A {\em set partition} of a set $S$ is a collection $\si$ of nonempty subsets whose disjoint union is $S$.  We write $\si=B_1/\dots/B_k\ptn S$ where the subsets $B_i$ are called {\em blocks}.  When no confusion will result, we often drop the curly braces and commas in the $B_i$.
For $[n]=\{1,\dots,n\}$, we use the notation
$$
\Pi_n=\{\si\ :\ \si\ptn[n]\}.
$$

To define pattern avoidance in this setting, suppose $\si=B_1/\dots/B_k\in\Pi_n$ and $S\sbe[n]$.   Then $\si$ has a corresponding {\em subpartition} $\si'$ whose blocks are the nonempty intersections $B_i\cap S$.  For example, if $\si=14/236/5\ptn[6]$ and $S=\{2,4,6\}$ then $\si'=26/4$.  We {\em standardize} a set partition with integral elements by replacing the smallest element by $1$, the next smallest by $2$, and so forth.  So the standardization of $\si'$ above is $13/2$.   Given two set partitions $\si$ and $\pi$, we say that $\si$ {\em contains $\pi$ as a pattern} if there is a subpartition of $\si$ which standardizes to $\pi$.  Otherwise we say that $\si$ {\em avoids} $\pi$.  Continuing our example, we have already shown that $\si=14/236/5$ contains $13/2$.  But $\si$ avoids $123/4$ because the only block of $\si$ containing three elements also contains the largest element in $\si$, so there can be no larger element in a separate block.
We let
$$
\Pi_n(\pi)=\{\si\in\Pi_n\ :\ \text{$\si$ avoids $\pi$}\}.
$$

In order to connect set partitions with the statistics of Wachs and White, we will have to convert them into restricted growth functions.  A {\em restricted growth function} (RGF) is a sequence $w=a_1\dots a_n$ of positive integers subject to the restrictions
\ben
\item  $a_1=1$, and
\item  for $i\ge2$ we have
\beq
\label{RGF}
a_i\le 1+\max\{a_1,\dots,a_{i-1}\}.
\eeq
\een
The number of elements in $w$ is called its {\em length} and we  let
$$
 R_n=\{w\ :\ \text{$w$ is an RGF of length $n$}\}.
$$
There is a simple bijection $\Pi_n\ra R_n$.  We say $\si=B_1/\dots/B_k\in\Pi_n$ is in {\em standard form} if $\min B_1<\dots<\min B_k$.  Note that this forces $\min B_1=1$.  We henceforth assume all partitions in $\Pi_n$ are written in standard form.  Associate with $\si$ the word $w(\si)=a_1\dots a_n$ where
$$
\text{$a_i = j$ if and only if $i\in B_j$.} 
$$
Using the example from the previous paragraph $w(\si)=122132$.  It is easy to see that $w(\si)$ is a restricted growth function and that the map $\si\mapsto w(\si)$ is the desired bijection.   It will be useful to have a notation for the RGFs of partitions avoiding a given pattern $\pi$, namely
$$
 R_n(\pi)=\{w(\si)\ :\ \si\in\Pi_n(\pi)\}.
$$ 

Sagan~\cite{sag:pas} described the set partitions in $\Pi_n(\pi)$ for each $\pi\in\Pi_3$.  Although it is not difficult to translate his  work into the language of restricted growth functions, we include the proof of the following  result for completeness and since it will be used many times subsequently.  Define the {\em initial run} of an RGF $w$ to be the longest prefix of the form $12\dots m$.  Also, we will use the notation $a^l$ to indicate a string of $l$ consecutive copies of the letter $a$ in a word.
Finally, say that $w$ is {\em layered} if $w=1^{n_1} 2^{n_2}\dots m^{n_m}$ for positive integers $n_1,n_2,\dots,n_m$.
\bth[\cite{sag:pas}]
\label{avoidance}
We have the following characterizations.
\ben
\item $ R_n(1/2/3) = \{w\in R_n\ :\ \text{$w$ consists of only $1$s and $2$s}\}$.
\item  $ R_n(1/23)=\left\{w\in R_n\ :\ \text{$w$ is obtained by inserting a single $1$ into a word}\right.$ 

\hspace{160pt}$\left.\text{of the form $1^l 23\dots m$ for some $l\ge0$ and $m\ge1$}\right\}.$
\item   $ R_n(13/2)=\{w\in R_n\ :\ \text{$w$ is layered}\}$.
\item $ R_n(12/3)=\{w\in R_n\ :\ \text{$w$ has initial run $1\dots m$ and $a_{m+1}=\dots=a_n\le m$}\}$.
\item $ R_n(123)=\{w\in R_n\ :\ \text{$w$ has no element repeated more than twice}\}$.
\een
\eth
\begin{proof}  In all cases it is easy to see that $ R_n(\pi)$ is contained in the right-hand side.  So we will concentrate on proving the other inclusions.

1.  If $w(\si)$ can have only $1$s and $2$s then $\si$ has at most two blocks and so avoids $1/2/3$.

2.   Suppose $\si\in\Pi_n(1/23)$ and let $m=\max\si$.  We assume $m>1$ since otherwise $w$ is clearly of the desired form.  Then no number $1<a\le m$ can be repeated since, if it were, the initial $1$ and two $a$'s in $w$ would correspond to a copy of $1/23$ in $\si$.  Similarly, there can not be two $1$'s in $w$ appearing after the $2$.  These two restrictions are equivalent to the description in the theorem.

3.  It suffices to show that if $a_i=r$ then $a_{i+1}=r$ or $r+1$ whenever $\si$ avoids $13/2$.  If not, then  $a_{i+1}=s$ where $s<r$ or $s>r+1$.  In the former case, $s$ must appear in $w$ in a position to the left of $a_i$ and the two copies of $s$ together with the copy of $r$ form a $13/2$ in $\si$.  If $s>r+1$, then $a_i$ can not be the first copy of $r$ in $w$.   Now these two copies of $r$ together with the $s$ give another contradiction.

4.   Suppose $\si$ avoids $12/3$ and consider $a_{m+1}$.  Condition~\ree{RGF} implies that $a_{m+1}=a_i\le m$ for some $i\le m$.  So if $a_{m+2}\neq a_{m+1}$, then the subpartition $i,m+1/m+2$ would be a copy of $12/3$ in $\si$ which is a contradiction.  Continuing in this way we obtain $a_{m+1}=\dots=a_n$.

5.  The given condition on $w$ implies that the blocks of $\si$ all have one or two elements.  Thus $\si$ avoids $123$.
\end{proof}

Using these characterizations, it is a simple matter to find the cardinalities of the avoidance classes.  
\bco[\cite{sag:pas}]
\label{avoidancecor}
We have the following cardinalities.
$$
\barr{l}
\#\Pi_n(1/2/3)=\#\Pi_n(13/2)=2^{n-1},\\[10pt]
\#\Pi_n(1/23)=\#\Pi_n(12/3)=\dil 1+\binom{n}{2},\\[10pt]
\#\Pi_n(123)=\dil\sum_{k\ge 0} {n\choose 2k} (2k)!!
\earr
$$
where $(2k)!!=(1)(3)(5)\dots(2k-1)$.\hqed
\eco

The four statistics of Wachs and White are denoted $\lb$, $\ls$, $\rb$, and $\rs$ where ``l" stands for ``left," ``r" stands for ``right," ``b" stands for ``bigger," and ``s" stands for ``smaller."  We will describe the left-bigger statistic and the other three should become clear by analogy.   Given a word $w=a_1\dots a_n$ we define
$$
\lb(a_j)=\#\{a_i\ :\ \text{$i<j$ and $a_i>a_j$}\}.
$$ 
In words, we count the set  of integers occuring before $a_j$ and bigger than $a_j$.  It is important to note that we are taking the cardinality of a set, so if there are multiple copies of such an integer then it is only counted once.  Also, clearly $\lb(a_j)$ depends on the word containing $a_j$, not just $a_j$ itself even though, for simplicity, our notation does not reflect that.   By way of example, if $w=1223142$ then
$\lb(a_5)=2$ since there is a $2$ and a $3$ to the left of $a_5=1$.  Finally, define
$$
\lb(w)=\lb(a_1)+\dots+\lb(a_n).
$$
Continuing our example,
$$
\lb(1223142)=0+0+0+0+2+0+2=4.
$$
To simplify notation, we will write $\lb(\si)$ for the more cumbersome $\lb(w(\si))$.  Our main objects of study will be the generating functions
$$
\LB_n(\pi)=\LB_n(\pi;q)=\sum_{\si\in\Pi_n(\pi)} q^{\lb(\si)}
$$
and the three analogous polynomials for the other statistics.  Often, we will even be able to compute the multivariate generating function
$$
F_n(\pi)=F_n(\pi;q,r,s,t)=\sum_{\si\in\Pi_n(\pi)} q^{\lb(\si)} r^{\ls(\si)} s^{\rb(\si)} t^{\rs(\si)}.
$$

The rest of this article is structured as follows.  In the next four sections we will compute $F_n(\pi)$ for $\pi=1/2/3,\ 1/23,\ 13/2$, and $12/3$, respectively.  In Section~\ref{sec123} we study $123$ which is a more difficult pattern to analyze.
One can also consider avoiding more than one pattern at once and this is the goal of Section~\ref{secmpa}.  Various results which did not fit earlier into the paper will be found in Section~\ref{secmr}.  Finally we end with some open problems and areas for future research.

%
%

\section{The pattern $1/2/3$}
\label{sec1/2/3}

We first consider the set partition $1/2/3$. We begin by presenting the four-variable generating function from which we derive the generating functions associated with each individual statistic.

\bth 
We have
$$F_n(1/2/3) = 1 + \sum \limits_{l=1}^{n-1}r^{n-l}s^l + \sum \limits_{l=2}^{n-1}\sum \limits_{k=0}^{n-l-1}\sum \limits_{i, j \geq 1}  \binom{n-i-j-k-2}{l-i-j} q^{l-i}r^{n-l}s^{l-\delta_{k,0} j}t^{n-l-k}$$
 where $\delta_{k, 0}$ is the Kronecker delta function.
\eth\label{F1/2/3}
\begin{proof}
By Theorem~\ref{avoidance}, any word $w \in  R_n(1/2/3)$ is composed solely  of ones and twos. Let $l$ denote the number of ones in $w$.  If such a word is weakly increasing, it is easy to see that these words contribute
$$1 + \sum \limits_{l=1}^{n-1}r^{n-l}s^{l}$$
to the generating function.  

Otherwise, let $w$ have at least one descent and $l$ ones. We can see that the word $w$ has the form $1^{i}w'1^{j}2^k$, where $i, j \geq 1$, the subword $w'$ begins and ends with a two, and $0\leq k \leq n-l-1$. 

For such $w$ the lb statistic is given by the number of ones after the first two, that is, by the number of ones not in $1^i$. Thus, $\lb(w) = l-i$. The ls statistic is given by the total number of twos in $w$, namely $n-l$. For the rb statistic, if $k$ is non-zero, then each one in $w$ contributes to the statistic.  Otherwise, only the ones that are not in $1^j$ contribute.  Combining the two cases gives  $\rb(w) = l-\delta_{k,0} j$. Finally, the rs statistic is given by the number of twos in $w'$, namely $n-l-k$. Putting all four statistics together produces
$$q^{\lb(w)}r^{\ls(w)}s^{\rb(w)}t^{\rs(w)} = q^{l-i}r^{n-l}s^{l-\delta_{k,0} j}t^{n-l-k}.$$

Choosing the number of ways of arranging the ones in $w'$ gives a coefficient of $$\binom{n-i-j-k-2}{l-i-j}.$$

Summing over $i, j, k, l$ and combining the cases gives our desired polynomial.
\end{proof}

The equations in the following corollary can be derived either by specialization of the four-variable generating function \ree{F1/2/3} and standard hypergeometric series techniques or by using the ideas in the proof of the previous result and ignoring the other three statistics.

\bco
\label{1/2/3cor}
We have
$$\LB_n(1/2/3) = \RS_n(1/2/3) = 1 + \sum \limits_{k = 0}^{n-2} \binom{n-1}{k+1}q^k,$$
and

\eqed{\LS_n(1/2/3) = \RB_n(1/2/3) = (r+1)^{n-1}.}
\eco

In view of the preceeding corollary, it would be nice to find explicit bijections $\phi:  R_n(1/2/3) \to  R_n(1/2/3)$ and $\psi:  R_n(1/2/3) \to  R_n(1/2/3)$ such that $\phi$ takes $\lb$ to $\rs$ and $\psi$ takes $\ls$ to $\rb$.  In the next two propositions, we present such bijections.

\bpr \label{bijection1/2/3}
There exists an explicit bijection $\phi: R_n(1/2/3) \to  R_n(1/2/3)$ such that for $v \in  R_n(1/2/3)$, 
$$\lb(v) = \rs(\phi(v)).$$
\epr
\begin{proof}
Let $v = a_1a_2\dots a_n \in  R_n(1/2/3)$. Define 
$$\phi(v) = a_1(3-a_n)(3-a_{n-1})\dots(3-a_3)(3-a_2).$$  
Because $v \in  R_n(1/2/3)$, by Theorem~\ref{avoidance}, it must be composed of only ones and twos and begin with a one.  It is clear that $\phi(v)$ has the same form, so $\phi$ is well defined. Also, $\phi$ is its own inverse and is therefore a bijection.  

If $\lb(v) = k$, then $v$ must contain a subword $v' = 21^k$ and no subword of the form $21^l$, with $l>k$.  In fact, this condition is clearly equivalent to $\lb(v)=k$.  It follows that $\phi(v') = 2^k1$ is a subword of $\phi(v)$ and $\phi(v)$ has no subword $2^l1$ with  $l>k$.  Therefore, $\rs(\phi(v)) = k = \lb(v)$, as desired.  
\end{proof}

\bpr
There exists an explicit bijection $\psi: R_n(1/2/3) \to  R_n(1/2/3)$ such that for $v \in  R_n(1/2/3)$, 
$$\ls(v) = \rb(\psi(v)).$$
\epr
\begin{proof}
Let $v \in  R_n(1/2/3)$.  If $v = 1^n$, then define $\psi(v) = v$. Clearly in this case $\ls(v) = 0 = \rb(v)$.

Otherwise, let $v = a_1a_2 \dots a_{i-1}a_i 1^{n-i}$ where $a_i = 2$ and $n-i \geq 0$.  Define 
$$\psi(v) = (3 - a_i)(3-a_{i-1}) \dots (3-a_2)(3-a_1)1^{n-i}.$$
The proof is now similar to that of Proposition~\ref{bijection1/2/3}, using the fact that the $1^{n-i}$ at the end of $v$ contributes to neither $\ls$ or $\rb$. 
\end{proof}

%
%

\section{The pattern $1/23$}
\label{sec1/23}

In this section we will determine $F_n(1/23)$, and thus the generating functions for all four statistics. We will find that $\lb$ and  $\rs$ are equal for any $w\in  R_n(1/23)$.

\begin{thm} \label{F1/23} We have \begin{equation}\label{F1/23eq} F_n(1/23)=(rs)^{\binom{n}{2}}+\sum_{m=1}^{n-1}\sum_{j=1}^{m}(qt)^{j-1}r^{\binom{m}{2}}s^{(n-m)(m-1)+m-j+\binom{m-1}{2}}.\end{equation}
\end{thm}

\begin{proof} 
If $\sigma$ avoids $1/23$ we know from Theorem~\ref{avoidance} that the associated RGF is obtained by inserting a single $1$ into a word of the form $1^l23\dots m$ for some $l\geq 0$ and $m\geq 1$. If  $l=0$ then the inserted $1$ must be at the beginning of the word in order for $w$ to be a RGF, so $w=12\dots n$. If $l>0$ then the inserted $1$ can be inserted after $j$ for any $1\leq j \leq m$, and the maximal letter $m$ satisfies $1\leq m \leq n-1$. If $w$ has maximal letter $m$ and we insert the $1$ after $j$ then $w$ is completely determined to be $1^{n-m}23\dots j1\dots m$.

In summary, either $w=12\dots n$ or $w$ is determined by the choice of $1\leq j \leq m$ and $1\leq m \leq n-1$. If $w=12\dots n$ then $\rb(w)=\ls(w)=\binom{n}{2}$ and $\lb(w)=\rs(w)=0$. For all other $w$ 
we have the following:

\begin{enumerate}
\item $\lb(w)=j-1$, 
\item $\ls(w)=\binom{m}{2}$, 
\item $\rb(w)=(n-m)(m-1)+m-j+\binom{m-1}{2}$
\item $\rs(w)=j-1$.
\end{enumerate}

1.  Only the inserted $1$ has elements which are left and bigger which are the numbers $2$ through $j$. So $\lb(w)=j-1$. 

2.  Since $w$ is an RGF every letter $i$ contributes $i-1$ to the $\ls$ giving a total of $\ls(w)=1+\dots +(m-1)=\binom{m}{2}$. 

3.  The first $n-m$ ones of $w$ each have $m-1$ elements which are right and bigger, so they contribute $(n-m)(m-1)$ to the $\rb$. The inserted $1$ has $m-j$ letters which are right and bigger.  Any element $i$ such that $2\leq i \leq m$ appears only once and contributes $m-i$ to the $\rb$. This means we have an additional $(m-2)+\dots +0 =\binom{m-1}{2}$. Hence $\rb(w)=(n-m)(m-1)+m-j+\binom{m-1}{2}$. 

4.  The only elements which have a number right and smaller are the elements $2$ through $j$, and the only number which is right and smaller of these elements is the inserted $1$. Hence $\rs(w)=j-1$. 

Summing over all the valid values for $m$ and $j$ gives us our equality. 
\end{proof}

The following result can be quickly seen by specializing Theorem~\ref{F1/23} or its demonstration,
so we have omitted the proofs.

\begin{cor}
\label{1/23cor}
We have $\lb(w)=\rs(w)$ for all words $w\in  R_n(1/23)$  and $$\LB_n(1/23)=\RS_n(1/23)=1+\sum_{j=1}^{n-1}(n-j)q^{j-1}.$$
Also
$$\LS_n(1/23)=r^{\binom{n}{2}}+\sum_{m=1}^{n-1}m r^{\binom{m}{2}},$$
and

 \eqed{\RB_n(1/23)=s^{\binom{n}{2}}+\sum_{m=1}^{n-1}\sum_{j=1}^{m}s^{(n-m)(m-1)+m-j+\binom{m-1}{2}}.}
\end{cor}

%
%

\section{The pattern 13/2 }
\label{sec13/2}

In this section, we begin by evaluating the four-variable generating function $F_n(13/2)$. Goyt and Sagan~\cite{gs:sps} have previously proven a theorem regarding the single-variable generating functions for the $\ls$ and $\rb$ statistics, and we will adapt their map and proof to obtain the multi-variate generating function for $13/2$. This generating function is closely related to integer partitions. A {\em partition $\lambda = (\lambda_1, \lambda_2, \dots, \lambda_k)$  of an integer $t$} is a weakly increasing sequence of positive integers such that $\sum_{i=1}^k \lambda_i = t$. The $\lambda_i$ are called \emph{parts}. Additionally, we will define an integer partition $n-\lambda = (n-\lambda_k,\dots n-\lambda_2, n-\lambda_1)$.  Let $|\lambda| = \sum_{i=1}^k \lambda_i$. We will denote by $D_{n-1}$  the set of integer partitions with distinct parts of size at most $n-1$. 

\bth
\label{13/2}
We have 
$$F_n(13/2) = \prod_{i=1}^{n-1} (1+r^{n-i} s^i).$$
\eth

\begin{proof}

Suppose $w \in  R_n(13/2)$. By Theorem~\ref{avoidance}, $w$ is layered and so  $\lb$ and $\rs$ are zero, resulting in no contribution to the generating function.  For the other two statistics, since $w$ is layered it has the form
$w=1^{n_1}2^{n_2} \dots m^{n_m}$ where $m$ is the maximum element of $w$.   Define $\phi:  R_n(13/2) \to D_{n-1}$ by
$$ \phi(w) = (\lambda_1, \lambda_2, \dots, \lambda_{m-1})$$
where $\lambda_j = \sum_{i=1}^{j}n_i$ for $1 \leq j \leq m-1$.  Note that since the $n_j$ are positive, the $\la_j$ are distinct, increasing,  and less than $n$ since the sum never includes $n_m$.  Thus the map is well defined.

We now show that $\phi$ is a bijection by constructing its inverse.  Given $\lambda = (\lambda_1, \lambda_2, \dots, \lambda_{m-1})$, consider for $1 \leq j \leq m$, the differences $n_j = \lambda_j - \lambda_{j-1}$, where we define $\lambda_0 = 0$ and $\lambda_m = n$. It is easy to see that sending $\la$ to $w=1^{n_1}2^{n_2} \dots m^{n_m}$ is a well-defined inverse for $\phi$.

We next claim that if $\phi(w)=\la$ then $\rb(w)=|\la|$.  Indeed, from the form of $w$ and $\la$ we see that
$$
\rb(w)=\sum_{i=1}^{m-1} n_i(m-i)=\sum_{j=1}^{m-1} \sum_{i=1}^{j}n_i = |\la|.
$$
Similarly we obtain $\ls(w)=|n-\la|$.  It follows that
$$
F_n(13/2)=\sum_{\lambda \in D_{n-1}}r^{|n-\lambda|}s^{|\lambda|}=\prod_{i=1}^{n-1} (1+r^{n-i} s^i)
$$
as desired.
\end{proof}

The generating function of each individual statistic is easy to obtain by specialization of Theorem~\ref{13/2} so we have omitted the proofs.

\bco[\cite{gs:sps}] 
\label{13/2cor} 
We have
$$\LB_n(13/2) = 2^{n-1} = \RS_n(13/2)$$
and

\eqed{\LS_n(13/2) = \prod_{i=1}^{n-1} (1 + q^i) = \RB_n(13/2).}
\eco

%
%

\section{The pattern $12/3$}
\label{sec12/3}

In this section, we determine $F_n(12/3)$.   The other polynomials associated with $12/3$ are obtained as corollaries. We find this avoidance class interesting because it leads to a connection with number theory.
\bth\label{F12/3} We have
 \begin{equation}\label{F12/3eq}
F_n(12/3) = r^{\binom{n}{2}}s^{\binom{n}{2}}+\sum_{m=1}^{n-1}\sum_{i=1}^{m}q^{(n-m)(m-i)}r^{\binom{m}{2}+(n-m)(i-1)}s^{\binom{m}{2}}t^{m-i}.
\end{equation}
\eth
\begin{proof} By Theorem~\ref{avoidance}, the elements of $ R_n(12/3)$ are the words of the form 
$$
w = 123\dots m i^{n-m}
$$ 
where $i\le m$. If $w=123\dots n$ then $\ls(w)=\rb(w)=\binom{n}{2}$ and $\lb(w)=\rs(w)=0$. Otherwise $m<n$. In this case, we will show the following:
\ben
\item $\lb(w) = (n-m)(m-i)$,
\item $\ls(w) = \binom{m}{2} + (n-m)(i-1)$, 
\item $\rb(w) = \binom{m}{2}$, 
\item $\rs(w) = m-i$.
\een

1. There are $n-m$ copies of $i$ in $w$ and these are the only elements contributing to $\lb$. Each $i$ has the elements $(i+1)(i+2)\dots m$ to its left that are bigger than it. So $\lb(i) = m-i$ for all $i$ and $\lb(w) = (n-m)(m-i)$.

2. Each element $w_j$ of $w$ has $\ls(w_j)=w_j-1$ by condition \ree{RGF}. Using this and the form of $w$ easily yields the desired equality. 

3.  This is similar to the previous case, noting that only the initial run of $w$ contributes to $\rb$.

4. We can see that the only elements $w_j$ with $\rs(w_j)>0$ will be those in the initial run such that $w_j>i$. These are precisely the elements $(i+1)(i+2)\dots m$ and each element has exactly one element to its right that is smaller than it. So $\rs(w) = m-i$. 

Summing over the valid values of $m$ and $i$, we have \ree{F12/3eq}.
\end{proof}

The next corollary follows easily by specialization of \ree{F12/3}. 
\bco\label{ls12/3} We have 
$$
\LS_n(12/3) = r^{\binom{n}{2}} + \sum_{m=1}^{n-1}\sum_{i=1}^{m}r^{\binom{m}{2}+ (n-m)(i-1)},
$$
and
$$
\RB_n(12/3)=s^{\binom{n}{2}}+ \sum_{m=1}^{n-1} ms^{\binom{m}{2}},
$$
as well as

\eqed{\RS_n(12/3)=1 + \sum_{k=0}^{n-2}(n-k-1)t^k.}
\eco

The coefficients of  $\LB_n(12/3)$ have an interesting interpretation.

\bpr We have 
\begin{equation}\label{lb12/3} 
\LB_n(12/3)=\sum_{k=0}^{\left\lfloor (n-1)^2/4 \right\rfloor} D_kq^k,
\end{equation}
where $D_k = \#\{d\ge1 : d\mid k \text{ and } d+ \frac{k}{d} +1 \leq n\}$.
\epr
\begin{proof}  Set $r=s=t=1$ in~\ree{F12/3eq}. We begin by showing the degree of $\LB_n(12/3)$ is $\left\lfloor (n-1)^2/4\right\rfloor$. 
By \ree{avoidance} we can let $w= 123\dots m i^{n-m}$ be a word in $ R_n(12/3)$. 

In order to maximize the $\lb(w)$, we can assume $i=1$. 
So, using the formula for $\lb(w)$ derived in the proof of Theorem~\ref{F12/3}, we must maximize $(n-m)(m-1)$. We take the derivative with respect to $m$ and set the equation equal to zero to obtain $n-2m+1=0$ 
and $m=\frac{n+1}{2}$. To get integer values of $m$, we obtain
\begin{equation}
\begin{cases}
m = \frac{n+1}{2} \mbox{ if } n \text{ is odd}, \\
m = \left\lceil \frac{n+1}{2}\right\rceil \mbox{ or } \left\lfloor\frac{n+1}{2}\right\rfloor \mbox{ if } n \text{ is even.}
\end{cases}
\end{equation}
In either case, the maximum value of $\lb$ is $\left\lfloor (n-1)^2/4\right\rfloor$.

We now show the coefficient of $q^k$ is $D_k$. As before, let $w= 123\dots m i^{n-m}$ be a word associated with a set partition that avoids $12/3$ and let $\lb(w)=k$. If we let $d=n-m$ be the number of $i$'s, it is clear that $\lb(w) = d(m-i) = k$ and therefore, $m-i = \frac{k}{d}$. Because $w$ must be of length $n$, we now must determine which divisors $d$ of $k$ are valid.  Each of the $d$ trailing $i$'s has $\frac{k}{d}$ elements to its left and bigger. Because $i\geq 1$, the leading one cannot be such an element.  Thus in order for $w$ to be of length $n$  we must have  $d+\frac{k}{d}+1 \leq n$.
\end{proof}

The above formulation of $\LB_n(12/3)$ leads to the following corollary, showing a connection to number theory.
\bco When $k\leq n-2$, we have $D_k=\tau(k)$, the number-theoretic  function which counts the divisors of $k$.
\eco
\begin{proof} We show that if $k\leq n-2$ then all positive divisors $d$ of $k$ are valid. We know that $d+\frac{k}{d}\leq k+1$ because $d=1$ and $d=k$ are the divisors of $k$ which maximize $d+\frac{k}{d}$. Thus, we have $d+\frac{k}{d}+1 \leq k+2 \leq n$. Therefore every positive divisor of $k$ satisfies the inequality in the definition of $D_k$, and this implies $D_k=\tau(k)$.
\end{proof}

%
%

\section{The pattern $123$}
\label{sec123}

The reader will have noticed that for the other four set partitions of [3], we provided a $4$-variable generating function describing all four statistics on the avoidance class of those partitions. 
The pattern $123$, however, is much more difficult to deal with and so we will content ourselves with results about the individual statistics. We will start with the left-smaller statistic.

\bth\label{thm:ls123} We have 
\begin{equation}\label{eq:ls123}
\LS_n(123) = \sum_{m = \ce{n/2}}^n \left [ \sum_{L} \left ( \prod_{g=1}^{n-m} (m-\ell_g + g) \right )q^{\binom{m}{2} + \sum\limits_{\ell \in L} (\ell - 1)} \right ]
\end{equation}
where the inner sum is over all subsets $L = \{\ell_1, \ell_2, \dots, \ell_{n-m} \}$ of $[m]$
 with $\ell_1 > \dots > \ell_{n-m}$.
\eth

\bprf

We start by noting that if a word has a maximum element $m$, then there must be $n-m$ repeated elements in the word, i.e., elements $i$ that appear after the initial occurrence of $i$. The bounds on our outer sum are given by the largest possible value of $m$ being $n$, and the smallest possible value of $m$ being $\ce{n/2}$, since we can repeat each element a maximum of two times.
We will now build our word $w$ by starting with a base sequence $12\dots m$ and adding in repeated elements. The base sequence will contribute $1 + 2 + \dots + (m-1) = \binom{m}{2}$ to $\ls(w)$. Let $L$ be the set of repeated elements we want to add to $w$. Then $L$ must contain $n-m$ elements from $[m]$, and since $w$ can have no element appear more than twice, $L$ can have no element appear more than once. 
For each element $\ell \in L$ that we add to our base sequence, we will increase $\ls(w)$ by $\ell - 1$.
So for any word $w$ with maximum $m$ formed in this way, we have $\ls(w) = \binom{m}{2} + \sum_{\ell \in L} (\ell - 1)$.

To find how many possible words can be so created, we start with our base sequence $12\dots m$, and build up our word by placing in the repeated elements from $L$ one at a time. There are $m - (\ell_1 - 1)$  spots where we can place the largest repeated element, $\ell_1$: anywhere after the original occurrence of $\ell_1$. Then when we place our second repeated element, $\ell_2$, we will have $m - (\ell_2 - 1) + 1$ spots, where the plus one comes from the extra space the first repeated element added in front of $\ell_2$. In general, when we place $\ell_g$ we will have $m - (\ell_g - 1) + (g-1) = m - \ell_g + g$ places to put it. The condition $\ell_1 > \dots > \ell_{n-m}$ is used since it implies that regardless of where $\ell_i$ is placed, one will have the same number of choices for the placement of $\ell_{i+1}$. 
Multiplying all these terms together and then summing over all possible subsets $L$ of $[m]$ gives us the coefficient of $q$. Finally, summing over all possible maximums of the words in the avoidance class gives us equation~\eqref{eq:ls123}.
\eprf

We were only able to find explicit expressions for certain coefficients of the polynomials generated from other statistics. We will now look at the left-bigger statistic.

\bth\label{lb123degree} We have the following.
\ben
\item The degree of $\LB_n(123)$ is $$\left\fl{\frac{n(n-1)}{6}\right}.$$

\item  The leading coefficient of $\LB_n(123)$ is
$$\begin{cases}
k! & \text{if } n = 3k \text{ or } 3k+1,\\
(k+2)k! & \text{if } n = 3k+2,\\
\end{cases}$$
for some nonnegative integer $k$.
\een
\eth

\bprf
We will show that a word of the form $w = 12\dots iw_{i+1}\dots w_n$ with $w_{i+1},\dots, w_n$ being a permutation of the interval $[1,n-i]$ will provide a maximum $\lb$ which is $\fl{(n(n-1))/6}$.

First we will prove that the elements after the initial run $12\dots i$ must be less than or equal to $i$. Note that, by definition of the initial run, $w_{i+1} \le i$. Now suppose, towards a contradiction, that for some $j \in [i+2,n]$, there was some element $w_j > i$. Then, since $w$ is an RGF, we must have $w_k = i+1$ for some $k \in [i+2,j]$. But by switching $w_k$ and $w_{i+1}$, we would increase $\lb$ by at least one since $w_{i+1} \le i$. So if any element after the initial run is greater than $i$, $\lb$ is not maximum.

Next we will show that the elements after the initial run have to be exactly those in the interval $[1,n-i]$, up to reordering. Suppose towards contradiction there was some element $t \in [1,n-i]$ that did not appear in the sequence after the initial run, and instead there appeared some element $s \in [n-i+1,i]$. Then $\lb(s) = i-s$. But $\lb(t) = i-t$, and since $s > t$, it follows that $\lb(t) > \lb(s)$. Therefore, if we want to maximize $\lb$, we must have the sequence after the initial run being exactly the interval $[1,n-i]$, up to reordering.

Now that we've established that our word is of the form $w = 12\dots iw_{i+1}\dots w_n$ with $w_{i+1},\dots ,w_n$ being exactly those elements in the interval $[1,n-i]$, we simply need to maximize $\lb$ using some elementary calculus. 
\begin{align}
\lb(w) &= (i - w_{i+1}) + (i - w_{i+2}) + \dots + (i - w_n) \nonumber\\
	&= (i - 1) + (i - 2) + \dots + (2i- n)\nonumber \\
	&= \frac{(4n+1)i - 3i^2 - n^2 - n}{2}\label{lbeqn}
\end{align}

Considering $i$ as a real variable and differentiating gives us a maximum value of $\lb(w)$ when $i = (4n+1)/6$. We must modify this slightly since we want $i$ to be integral. Rounding $i$ to the closet integer gives
$$i = \begin{cases}
\left \lfloor \frac{4n+1}{6}\right\rfloor & \text{if } n = 3k,\\
\left \lceil \frac{4n+1}{6}\right\rceil & \text{if } n = 3k+1,\\
\left \lfloor \frac{4n+1}{6}\right\rfloor \text{or }\left \lceil \frac{4n+1}{6}\right\rceil & \text{if } n = 3k+2,\\
\end{cases}$$ 
for some nonnegative integer $k$.

Plugging each value of $n$ and $i$ back into equation~\eqref{lbeqn} gives us an $\lb$ of $\fl{(n(n-1))/6}$ in all cases. As we've mentioned before, the elements $w_{i+1}, \dots, w_n$ must be exactly those in the interval $[1,n-i]$, but the ordering doesn't matter. This means the leading coefficient of $\LB_n(123)$ will be precisely the number of ways to permute the $n-i$ elements after the initial run. This gives us our second result.
\eprf

Our next theorem will involve the Fibonacci numbers. Recall that the $n$th Fibonacci number $F_n$ is defined recursively as 

\begin{equation}\label{fib}
F_n = F_{n-1} + F_{n-2}
\end{equation}
with initial conditions $F_0 = 1$ and $F_1 = 1$.

\bth\label{lbcoeffs} We have the following coefficients.
\ben
\item The constant term of $\LB_n(123)$ is $F_n$.

\item The coefficient of $q$ in $\LB_n(123)$ is $(n-2)F_{n-2}$.
\een
\eth

\bprf 
	If $\lb(\sigma) = 0$, then $w = w(\sigma)$ must be layered. Let $L(n)$ be the set of layered words $w(\sigma)$ with $\sigma \in \Pi_n(123)$. It follows that the constant term of $\LB_n(123)$ is $\#L(n)$. Define $L_i(n) = \{w \in L(n)\ |\ w \text{ starts with } i \text{ ones}\}$. 
Then $\#L(n) = \#L_1(n) + \#L_2(n)$. But $\#L_i(n) = \#L(n-i)$ for $i = 1,2$, since if $w$ begins with $i$ ones then the rest of the word is essentially a layered word with $n-i$ elements.  Therefore, $\#L(n) = \#L(n-1) + \#L(n-2)$. Since $\#L(0) = 1$ and $\#L(1) = 1$, we have $\#L(n) = F_{n}$. 

To prove the second claim, let $w \in  R_n(123)$ with $\lb(w) = 1$. Then  there must be exactly one descent in $w$ and it must be of the form $w_{j+1}=w_j-1$ for some $2\le j\le n- 1$.  Removing $w_j$ and $w_{j+1}$ from $w$ and then subtracting one from all $w_k$ with $k>j+1$ gives an element $w'\in R_{n-2}$ which is layered.  So, from the previous paragraph, there are $F_{n-2}$ choices for $w'$.  Further, there were $n-2$ choices for $j$ and so the total number of $w$ is $(n-2)F_{n-2}$.
\eprf

We will now look at the right-smaller statistic.

\bth\label{RSdegree/con} We have the following.

\ben

\item The degree of $\RS_n(123)$ is $$\left\fl{\frac{(n-1)^2}{4}\right}.$$

\item The leading coefficient of $\RS_n(123)$ is $1$ when $n$ is odd, and $2$ when $n$ is even.

\item The constant term of $\RS_n(123)$ is $F_n$.

\een
\eth

\bprf
The proof of the first result is very similar to the proof of the degree of $\LB_n(123)$. When looking at the right-smaller statistic, the word that maximizes $\rs$ is of the form $w = 12\dots i (n-i) \dots 21$, where $12\dots i$ is the initial run. Calculating $\rs(w)$ gives 
\beq\label{maxrs}
\rs(w) = (n-i)(i-1),
\eeq
and differentiating with respect to the real variable $i$ and maximizing gives $i = (n+1)/2$. Since we want $i$ to be integral, we have 
$$i = \begin{cases}
 \frac{n+1}{2} & \text{if } n \text{ is odd,}\\
\left \lfloor \frac{n+1}{2}\right\rfloor \text{or }\left \lceil \frac{n+1}{2}\right\rceil & \text{if } n \text{ is even.}\\
\end{cases}$$ 
Plugging each value of $i$ and $n$ into \eqref{maxrs} gives $\fl{(n-1)^2/4}$ in both cases.  Also, the number of choices for $i$ gives the leading coefficient of $\RS_n(123)$. 

The proof for  the constant term of $\RS_n(123)$ is the same as for $\LB_n(123)$  since for any $w$ we have $\rs(w) = 0$ if and only if $\lb(w) = 0$.
\eprf

Our final result of this section gives the degree of $\RB_n(123)$.   It follows immediately from the easily proved fact that  the word which maximizes $\rb$ is  $w = 12\dots n$.

\bth\label{rbdegree}
$\RB_n(123)$ is monic and has degree $\binom{n}{2}.$\hqed
\eth

\section{Multiple pattern avoidance}
\label{secmpa}

\begin{table}
\begin{center}
\scalebox{1.1}{
\begin{tabular}{ |c|c| }
 \hline
  \multicolumn{1}{|c|}{Avoidance Class} & \multicolumn{1}{|c|}{Associated RGFs}  \rule{0pt}{20pt}\\ [5pt]
  \hline 
 $\Pi_n(1/2/3, 1/23)$ & $1^n,\ 1^{n-1}2,\ 1^{n-2}21$   \rule{0pt}{20pt}\\ [5pt]
  \hline
  $\Pi_n(1/2/3, 13/2)$ & $1^m2^{n-m}$ for all $1 \leq m \leq n$  \rule{0pt}{20pt}\\ [5pt]
  \hline
  $\Pi_n(1/2/3, 12/3)$ & $1^n,\ 12^{n-1},\ 121^{n-2}$  \rule{0pt}{20pt}\\ [5pt]
  \hline
 $\Pi_n(1/23, 13/2)$ & $1^{n-m+1}23 \dots m$ for all $1 \leq m \leq n$  \rule{0pt}{20pt}\\ [5pt]
  \hline
  $\Pi_n(1/23, 12/3)$ & $1^n,\ 12 \dots (n-1)1,\ 12 \dots n$  \rule{0pt}{20pt}\\ [5pt]
  \hline
 $\Pi_n(1/23, 123)$ & $12 \dots n,\ 12 \dots (n-1)$ with an additional $1$ inserted  \rule{0pt}{20pt}\\ [5pt]
  \hline
 $\Pi_n(13/2, 12/3)$ & $12 \dots m^{n-m+1}$ for all $1 \leq m \leq n$ \rule{0pt}{20pt}\\ [5pt]
  \hline
 $\Pi_n(13/2, 123)$ & layered RGFs with at most two elements in each layer  \rule{0pt}{20pt}\\ [5pt]
  \hline
  $\Pi_n(12/3, 123)$ & $12 \dots (n-1)m$ for all $1 \leq m \leq n$  \rule{0pt}{20pt}\\ [5pt]
  \hline

\end{tabular}}\\
\caption{Avoidance classes avoiding two partitions of $[3]$ and associated RGFs}
\label{table:avoidance elements}
\end{center}
\end{table}

Rather than avoiding a single pattern, one can avoid multiple patterns.  Define, for any set $P$ of set partitions
$$
\Pi_n(P)=\{\si\in\Pi_n\ :\ \text{$\si$ avoids every $\pi\in P$}\}.
$$
Similarly adapt the other notations we have been using.  
Goyt~\cite{goy:apt} characterized that cardinalities of $\Pi_n(P)$ for any  $P\sbe\fS_3$.
Our goal in this section is to do the same for $F_n(P)$.  We will not include those $P$ containing both $1/2/3$ and $123$ since it is easy to see from Theorem~\ref{avoidance} that there are no such partitions for $n\ge5$.

Table~\ref{table:avoidance elements} shows the avoidance classes and the resulting restricted growth functions that arise from avoiding two patterns of length $3$.  These as well as the entries in Table~\ref{table:threefour patterns} also appear in Goyt's work, but we include them here for completeness.
 For ease of references, we give a total order to $\Pi_3$ as follows
\beq
\label{order}
1/2/3,\ 1/23,\ 13/2,\ 12/3,\ 123
\eeq
and list the elements of any set $P$ in lexicographic order with respect to~\ree{order}.  Finally, for any $P\sbe \Pi_3$ we have $\Pi_n(P)=\Pi_n$ for $n<3$.  So we assume for the rest of this section that $n\ge3$.

The next result translates this table into generating functions.  This is routine and only uses techniques we have seen in earlier sections so the proof is omitted.

\bth \label{F(1/2/3,13/2)}
For $n\ge3$ we have
\begin{align*}
F_n(1/2/3, 1/23) &= 1 + rs^{n-1} + qrs^{n-2}t, \\
F_n(1/2/3, 13/2) &= 1+ \sum_{i=1}^{n-1}r^is^{n-i}, \\
F_n(1/2/3, 12/3) &= 1+rs^{n-1} + q^{n-2}rst, \\
F_n(1/23, 13/2) &= 1 + \sum_{i=1}^{n-1} r^{\binom{n-i+1}{2}}s^{\binom{n}{2}-\binom{i}{2}}, \\
F_n(1/23, 12/3) &= 1 + (qt)^{n-2}(rs)^{\binom{n-1}{2}} + (rs)^{\binom{n}{2}}, \\
F_n(1/23, 123) &= (rs)^{\binom{n}{2}} + r^{\binom{n-1}{2}} \sum_{i=0}^{n-2}(qt)^is^{\binom{n}{2}-i-1}, \\
F_n(13/2, 12/3) &= 1 + \sum_{i=1}^{n-1}r^{\binom{n}{2}-\binom{i}{2}}s^{\binom{n-i+1}{2}}, \\
F_n(13/2, 123) &=  \sum_{m = \ce{n/2}}^n \left [ \sum_{L} r^{\binom{m}{2} + \sum\limits_{\ell \in L} (\ell - 1)} s^{\binom{m}{2} + \sum\limits_{\ell \in L} (m - \ell)} \right ], \\  
F_n(12/3, 123) &= ( rs)^{\binom{n}{2}} + s^{\binom{n-1}{2}} \sum_{i=0}^{n-2} (qt)^ir^{\binom{n}{2}-i-1},
\end{align*}
where $L$ and $m$ in $F_n(13/2, 123)$ are defined as in Theorem~\ref{thm:ls123}.\hqed
\eth

Note that from this theorem we immediately get the following nice equidistribution results.
\bco
\label{multicor}
Consider the generating function $F_n(P)$ where $P\sbe\Pi_3$.  
\ben
\item We have $F_n(P)$ invariant under switching $q$ and $t$ if $13/2\in P$ or $P$ is one of 
$$
\{1/2/3, 1/23\}; \ \{1/23, 12/3\}; \ \{1/23, 123\}; \ \{12/3, 123\}.
$$
\item We have $F_n(P)$ invariant under switching $r$ and $s$ if $P$ is one of 
$$
\{1/2/3, 13/2\}; \   \{1/23, 12/3\}.
$$
\item  We have the following equalities between generating functions for different $P$:
$$
F_n(1/23, 13/2;q,r,s,t)=F_n(13/2, 12/3;q,s,r,t)
$$
and

\eqed{F_n(1/23, 123;q,r,s,t)=F_n(12/3, 123;q,s,r,t).}
\een
\eco

Next, we will examine the outcome of avoiding three and four partitions of $[3]$. We can see the avoidance classes and the resulting restricted growth functions  in Table~\ref{table:threefour patterns}. 
 The entries in this table can easily be turned into a polynomial by the reader if desired. Avoiding all five partitions of $[3]$  is not included because it would contain both $1/2/3$ and $123$.

\begin{table}[h]
\begin{center}
\scalebox{1.1}{
\begin{tabular}{ |c|c| }
 \hline
  \multicolumn{1}{|c|}{Avoidance Class} & \multicolumn{1}{|c|}{Associated RGFs}  \rule{0pt}{20pt}\\ [5pt]
  \hline 
 $\Pi_n(1/2/3, 1/23, 13/2)$ & $1^n$, \ $1^{n-1}2$  \rule{0pt}{20pt}\\ [5pt]
  \hline
  $\Pi_n(1/2/3, 1/23, 12/3)$ & $1^n$, \ $121$ when $n=3$  \rule{0pt}{20pt}\\ [5pt]
  \hline
 $\Pi_n(1/2/3, 13/2, 12/3)$ & $1^n$, \ $12^{n-1}$  \rule{0pt}{20pt}\\ [5pt]
  \hline
 $\Pi_n(1/23, 13/2, 12/3)$ & $1^n, \ 12 \dots n$   \rule{0pt}{20pt}\\ [5pt]
  \hline
 $\Pi_n(1/23, 13/2, 123)$ & $1^22 \dots (n-1), \ 12 \dots n$  \rule{0pt}{20pt}\\ [5pt]
  \hline
  $\Pi_n(1/23, 12/3, 123)$ & $12 \dots (n-1)1, \ 12 \dots n$  \rule{0pt}{20pt}\\ [5pt]
  \hline
  $\Pi_n(13/2, 12/3, 123)$ & $12 \dots (n-2)(n-1)^2$, \ $12 \dots n$  \rule{0pt}{20pt}\\ [5pt]
  \hline\hline
 $\Pi_n(1/2/3, 1/23, 13/2, 12/3)$ & $1^n$  \rule{0pt}{20pt}\\ [5pt]
  \hline
  $\Pi_n(1/23, 13/2, 12/3, 123)$ & $12 \dots n$   \rule{0pt}{20pt}\\ [5pt]
  \hline

\end{tabular}}\\
\caption{Avoidance classes and associated RGFs avoiding three and four partitions of $[3]$ }
\label{table:threefour patterns}
\end{center}
\end{table}

\section{Miscellaneous Results}
\label{secmr}

In this section we  present other interesting results we have found. These include theorems regarding longer patterns and several bijections. 
We will start with a sequence of results concerning the pattern $14/2/3$.

Our first theorem concerns applying the $\lb$ statistic to the avoidance class of $14/2/3$, from which a connection arises between $14/2/3$ avoiding set partitions and integer compositions. First, we characterize $ R_n(14/2/3)$. We define the index $i$ to be a \emph{dale of height $a$} in $w$ if $a_i=a$ and
$$
a_i=\max\{a_1,\dots,a_{i-1}\}-1.
$$

\ble
\label{14/2/3_avoid}
For an RGF $w$, $w$ is contained in $ R_n(14/2/3)$ if and only if $w$ meets the following restrictions:
\begin{itemize}
\item for $i\geq 2$ we have $a_i\geq \max\{a_1,\dots,a_{i-1}\}-1$, and
\item if $w$ has a dale of height $a$, then $w$ does not have a dale of height $a+1$. 
\end{itemize}
\ele

\bprf
Let $\sigma$ avoid $14/2/3$. Assume, towards contradiction, that there existed an $a_i$ in $w=w(\sigma)$ with $a_i<\max\{a_1,\dots,a_{i-1}\}-1$ and let $a=a_i$. By the structure of restricted growth functions, this implies that $a(a+1)(a+2)a$ exists as a subword in $w$. But then these four elements  give rise to an occurence of $14/2/3$ in $\sigma$, which is a contradiction. This shows the first inequality.
Now assume that there existed dales of height $a$ and height $a+1$ in $w$. This would require $w$ to contain $(a+1)a(a+2)(a+1)$ as a subword, which again implies an occurance of $14/2/3$ in $\sigma$. This shows the height requirement for dales. 

Now assume that $\sigma$ is a partition with $w=w(\sigma)$ meeting the listed requirements. If $\sigma$ contained $14/2/3$ as a pattern, then $abca$ must occur as a subword in $w$, with $a\neq b\neq c$. If $a$ was the minimum value in this subword, then either $a<b-1$ or $a<c-1$, which contradicts the first restriction put on $w$ in view of the second $a$ in the subword. Further, if $a$ was the maximum value in this subword, then either $b<a-1$ or $c<a-1$, raising the same contradiction in view of the second $a$. Similarly, we can rule out $c<a<b$. Thus the only remaining possibility is that $b<a<c$. By the first condition in the lemma, it then must be that the subword is exactly $a(a-1)(a+1)a$, which contradicts the restriction on dales. Thus $\sigma$ avoids $14/2/3$, showing the reverse implication. 
\eprf

Note that a dale in a word $w$ contributes exactly one to $\lb(w)$.
And by the previous lemma, dales are the only source of $\lb$ for words in $ R_n(14/2/3)$. For the proof of our theorem about $\LB(14/2/3)$ we will also need the following notion: call $i$ a \emph{left-right maximum of value $a$} in $w$ if $a_i=a$ and
$$
a_i>\max\{a_1,\dots,a_{i-1}\}.
$$
Being an RGF is equivalent to having left-right maxima of values $1,2,\dots,m$ for some $m$. 

\bth
\label{LB(14/2/3)}
For $n\geq1$, we have 
$$
\LB_n(14/2/3)=2^{n-1}+\sum_{k=1}^{n-2}\left[\sum_{m\geq2}\binom{n-1}{k+m-1}\sum_{j\geq1}\binom{k-1}{j-1}\binom{m-j}{j}\right]q^k.
$$
\eth

\bprf
It is easy to see that the constant term in this polynomial comes from the layered partitions of $[n]$, all of which avoid $14/2/3$. Now consider the coefficient of $q^k$ for $k\geq 1$. From the discussion before the statement of the theorem, for a word in $ R_n(14/2/3)$ to have an $\lb$ of $k$, it must have $k$ dales. Further, we know that $i=1$ is always a left-right maximum of value $1$ in any RGF, and that $i=1$ is never a dale. It follows by Lemma~\ref{14/2/3_avoid}  that, to completely characterize an RGF of $\lb$ equal to $k$ and maximum value $m$ in $ R_n(14/2/3)$, it suffices to specify the remaining $m-1$ left-right maxima and the $k$ dale indices. As such, there are $\binom{n-1}{m+k-1}$ ways to choose a set $I$ which is the union of these two index sets. 

Let $I=\{i_1<i_2<\dots<i_{m+k-1}\}$ be such a set. We will indicate indices chosen for dales by coloring them blue, and left-right maxima by coloring them red. We define a \emph{run}  to be a maximal sequence of indices $i_c,i_{c+1},\dots,i_d$ which is monochromatic. Let $j$ be the number of blue runs, and let $b_s$ be the number of indices in the $s$th blue run, for $1\leq s\leq j$. As these numbers count the dales in $w$, we must have 
$$
b_1+b_2+\dots+b_j=k,
$$
or equivalently that $b_1,\dots,b_j$ form an integer composition of $k$. Thus there are $\binom{k-1}{j-1}$ ways of choosing $j$ blue runs. 

Now note that $I$ must start with a red run, and can end with either a red or blue run. Thus there are $j$ or $j+1$ red runs. Let $r_t$ be the length of the $t$th red run, for $1\leq t\leq j+1$, where we set $r_{j+1}=0$ if there are $j$ red runs. Furthermore, by the dale height restriction in Lemma~\ref{14/2/3_avoid}, we have $r_t \geq 2$ for $2\leq t\leq j$. Now as before, we have 
$$
r_1+r_2+\dots+r_{j+1}=m-1,
$$
subject to $r_1\geq 1$, $r_2,\dots,r_j\geq 2$, and $r_{j+1}\geq 0$. Using a standard composition manipulation, we can put this sum in correspondence with a composition of $m-j+1$ into $j+1$ parts, which gives $\binom{m-j}{j}$ ways to choose the red runs. Putting everything together and summing over the possible values of $m$ and $j$ gives the coefficient of $q^k$ as 
$$
\sum_{m\geq2}\binom{n-1}{k+m-1}\sum_{j\geq1}\binom{k-1}{j-1}\binom{m-j}{j}.
$$

All that is left is to give appropriate bounds for $k$. It follows by Lemma~\ref{14/2/3_avoid} that $w=121^{n-2}$ is in $ R_n(14/2/3)$ and that $w$ gives a maximizing $\lb$ of $n-2$. This gives $1\leq k\leq n-2$, and provides the correct parameters for the polynomial. 
\eprf

From the previous theorem, and from the characterization of $ R_n(14/2/3)$, several corollaries follow. 

\bco
\label{LB(14/2/3)=RS(14/2/3)}
We have 
$$
\LB_n(14/2/3)=\RS_n(14/2/3).
$$
\eco

\bprf
We proceed by finding a bijection $\phi$ that takes $ R_n(14/2/3)$ to itself, and that takes the $\lb$ statistic to the $\rs$ statistic. Let $w$ be a member of $ R_n(14/2/3)$. From Lemma~\ref{14/2/3_avoid}, we can partition $w$ into sections based on the dales of $w$. Specifically, let $a_i$ be a letter in $w$, and let $a=a_i$. If there is no dale of height $a$ or $a-1$ in $w$, then it follows that every copy of $a$ is adjacent in $w$. That is to say, we can break $w$  into
$$
w=w_1 a^l w_2,
$$
with $a_j<a$ for all $a_j$ in $w_1$, and $a_k>a$ for all $a_k$ in $w_2$. Call such a string a {\em plateau} of $w$. It follows that plateaus in $w$ contribute nothing to $\lb(w)$ or $\rs(w)$. We will let $\phi$ act trivially on the plateaus of $w$. 

If this is not the case, then there is a dale of height $a$ or $a-1$ in $w$. 
By Lemma~\ref{14/2/3_avoid} again, both $a$ and $a-1$ can not be dale heights.  So suppose $a-1$ is a dale height.
It follows that the occurances of $a$ and $a-1$ in $w$ are adjacent and we have 
$$
w=w_1 (a-1)^{l_0}a^{j_1}(a-1)^{l_1}\dots a^{j_t}(a-1)^{l_t} w_2, 
$$
with $l_0,\dots,l_{t-1}>0$, $l_t\geq 0$, and $j_1,\dots,j_t>0$. Further, we have $a_j<a-1$ for all $a_j$ in $w_1$, and $a_k>a$ for all $a_k$ in $w_2$. Such a string will be called a {\em dale section} of $w$. Breaking up $w$ in this manner shows that such a dale section contributes $l_1+\dots+l_t$ to $\lb(w)$, and either $j_1+\dots+j_{t-1}$ or $j_1+\dots+j_{t}$ to $\rs(w)$, depending on whether or not $l_t=0$. As such, if 
$$
d=(a-1)^{l_0}a^{j_1}(a-1)^{l_1}\dots a^{j_t}(a-1)^{l_t}
$$
is a dale section in $w$, we let 
$$
\phi(d)=
\begin{cases}
\begin{array}{ll}
(a-1)^{l_0}a^{l_1}(a-1)^{j_1}\dots a^{l_t}(a-1)^{j_t} &\text{ if }l_t>0, \\
(a-1)^{l_0}a^{l_1}(a-1)^{j_1}\dots a^{l_{t-1}}(a-1)^{j_{t-1}}a^{j_t} &\text{ if }l_t=0.
\end{array}
\end{cases}
$$
It follows that $\phi$ exchanges $\lb$ and $\rs$ for a dale section.

Now by the nature of $ R_n(14/2/3)$, we know that $w$ is merely a concatenation of plateaus and dale sections. Having defined $\phi$ on these parts of $w$, we define $\phi(w)$ by applying $\phi$ to the plateaus and dale sections of $w$ in a piecewise manner. It follows that $\phi$ is a bijection, since it is an involution. Finally, since $\lb(w)$ and $\rs(\phi(w))$ are sums over the dale sections of $w$ and $\phi(w)$, and since $\phi$ exchanges the two statistics on each dale section, it follows that we have $\lb(w)=\rs(\phi(w))$. 
\eprf

\bco
\label{14/2/3cor}
For $t\geq2$, we have
$$
\LB_n(14/2/3,1/2/\dots/t)=\sum_{i=0}^{t-2}\binom{n}{i}+ \sum_{k=1}^{n-2} \left[\sum_{m=2}^{t-1} \binom{n-1}{k+m-1}\sum_{j\geq1}\binom{k-1}{j-1}\binom{m-j}{j}\right]q^k
$$
and the equality
$$
\LB_n(14/2/3,1/2/\dots/t)=\RS_n(14/2/3,1/2/\dots/t).
$$
\eco

\bprf
Avoiding $1/2/\dots/t$ as well as $14/2/3$ adds the restriction that words must have maximum value less than or equal to $t-1$. Following the proof of Theorem~\ref{LB(14/2/3)} with this additional restriction gives the generating function $\LB_n(14/2/3,1/2/\dots/t)$. 

Next, we note that the same bijection from Corollary~\ref{LB(14/2/3)=RS(14/2/3)} also provides a bijection from $ R_n(14/2/3,1/2\dots/t)$ to itself, since $\phi$ preserves maximum values. The same map then ensures the second equality. 
\eprf

\bco
The polynomial $\LB_n(14/2/3,123)$ has degree $\lfloor n/3\rfloor$ and leading coefficient equal to
$$
\begin{cases}
\begin{array}{ll}
1& \text{ if } n=3k, \\
n& \text{ if } n=3k+1, \\
\frac{3n^2-7n+14}{6}& \text{ if } n=3k+2,
\end{array}
\end{cases}
$$
for some integer $k$. 
\eco

\bprf
Avoiding the pattern $123$ as well as $14/2/3$ adds the restriction that letters can be repeated at most twice in a word. Adapting the notation used in the proof of Corollary~\ref{LB(14/2/3)=RS(14/2/3)}, this implies that, for $w\in R_n(14/2/3,123)$, the dale sections of $w$ must have length equal to $3$ or $4$. Further, these dale sections can only contribute $1$ to $\lb(w)$. Thus to maximize $\lb(w)$, we maximize the number of dale sections contained in $w$. It follows from the restrictions on $w$ that this leads to a maximum of $\lfloor n/3\rfloor$. 

We now move to the leading coefficient. If $n=3k$ for some integer $k$, then it is clear that the only RGF $w$ in $ R_n(14/2/3,123)$ that achieves this maximum is 
$$
w=121343\dots(2k-1)2k(2k-1),
$$
giving a leading coefficient of $1$. 

Now let $w\in R_n(14/2/3,123)$ for $n=3k+1$. It follows that $w$ either has one dale section of length $4$, or one plateau of length $1$. In the first case, we note that a dale section of length $4$ has the form $a(a+1)(a+1)a$ or $a(a+1)a(a+1)$. As there will be $k$ total dales in $w$, we have $k$ choices for which dale section to extend, and $2$ choices for how to extend it. This gives $2k$ possible words of the first form. Now assume $w$ has a plateau of length $1$. Note that, once the index of this plateau has been chosen, the rest of the word is uniquely determined. As such, we can choose to place the plateau directly in front of any of the $k$ dale sections, or after the last dale section in $w$. This gives $k+1$ possible words of the second form. Summing over both possibilities now gives a leading coefficient of $n=3k+1$. 

Finally, we have $w\in R_n(14/2/3,123)$ for $n=3k+2$. There are four distinct possibilities for $w$ in this case. First, $w$ could contain one plateau of length $2$.  This gives $k+1$ possibilities as in the previous paragraph. The second possibility is that $w$ contains two plateaus of length $1$. If these plateaus are adjacent, then as in the previous case we have $k+1$ possibilities. Otherwise, we choose $2$ distinct places from these options, giving $\binom{k+1}{2}$ more words. In the third case, $w$ contains one plateau of length $1$ and one dale section of length $4$. We have $k+1$ choices for the plateau, and $2k$ possibilities for the dale section, giving $2k(k+1)$ words of this form. Finally, $w$ could contain two dale sections of length $4$. In this case, we choose two dale sections to extend. As there are two distinct ways to extend each dale section, this gives $4\binom{k}{2}$ such words. Summing over these four cases and using the substitution $n=3k+2$ gives the final result. 
\eprf

Our last corollary regarding the pattern $14/2/3$ involves multiple pattern avoidance with two partitions of $[4]$. First, we need a lemma. 

\ble
\label{RGF(14/2/3,13/2/4)}
For an RGF $w$, $w$ is contained in $ R_n(14/2/3,13/2/4)$ if and only if $w$ meets the following restrictions:
\begin{itemize}
\item For $n\geq2$ we have $a_i\geq \max\{a_1,\dots, a_{i-1}\}-1$, and
\item If $i$ is a dale of height $a$, then $a_j=a$ or $a_j=a+1$ for all $j>i$. 
\end{itemize}
\ele

\bprf
First, let $\sigma$ avoid $14/2/3$ and $13/2/4$, and let $w=w(\sigma)$. Since $ R_n(14/2/3,13/2/4)$ is a subset of $ R_n(14/2/3)$, the first inequality follows from Lemma~\ref{14/2/3_avoid}. Now assume that $i$ is a dale of height $a$ in $w$, and assume towards a contradiction that there exists $a_j$ in $w$ with $j>i$, $a_j\neq a$, and $a_j\neq a+1$. From the first inequality, it must be that $a_j>a+1$. Because $w$ is an RGF, it follows that $a(a+1)a(a+2)$ exists as a subword in $w$. But now these four elements will cause an occurance of $13/2/4$ in $\sigma$, which is a contradiction. 

For the reverse implication, let $\sigma$ be a partition with $w=w(\sigma)$ satisfying the above restrictions. From Lemma~\ref{14/2/3_avoid}, it follows that $\sigma$ will avoid $14/2/3$. To see that $\sigma$ will also avoid $13/2/4$, note that if $\sigma$ contained $13/2/4$, then the subword $abac$ would exist in $w$, with $a\neq b\neq c$. Using the first inequality, we can rule out all cases except $b<a<c$. But, as this implies a dale of height $b$ in $w$, this would lead to a contradiction with respect to the second restriction put on $w$ by the lemma. Thus $\sigma$ must also avoid $13/2/4$.
\eprf

\bco
We have 
$$
\LB_n(14/2/3,13/2/4)=2^{n-1}+\sum_{k=1}^{n-2}\left[\sum_{m\geq 2}\binom{n-1}{k+m-1}\right]q^k
$$
and 
$$
\LB_n(14/2/3,13/2/4)=\RS_n(14/2/3,13/2/4).
$$
\eco

\bprf
Following the proof of Theorem~\ref{LB(14/2/3)}, we note that the constant term in this polynomial comes from the layered partitions of $[n]$. Now consider a word $w$ in $ R_n(14/2/3,13/2/4)$ with $\lb$ equal to $k$ and maximum value $m$, for $k\geq 1$. From the previous lemma, it follows that the $k$ dales in $w$ must come to the right of the $m$ left to right maxima in $w$. As the leading one in $w$ provides the first left to right maximum, it suffices to choose $k+m-1$ other indices where we place the remaining left to right maxima in the left-most $m-1$ indices, and the $k$ dales afterwards. This gives $\binom{n-1}{k+m-1}$ such words, and summing over all possible values of $m$ gives the coefficient of $q^k$ for $k\geq1$. 

Finally, we note that the bijection from Corollary~\ref{LB(14/2/3)=RS(14/2/3)} also takes $ R_n(14/2/3,13/2/4)$ to itself. This gives the second equality. 
\eprf

For our final result, 
we provide two interesting relationships between the avoidance classes $\Pi(1/23)$ and $\Pi(12/3)$.

\bpr
\label{mpprop}
For $n\geq0$, we have the following equalities:
\begin{align*}
\LB_n(1/23)&=\RS_n(12/3),\\
\LS_n(1/23)&=\RB_n(12/3).
\end{align*}
\epr

\bprf
We will prove this theorem by providing a bijection that maps from $ R_n(1/23)$ to $ R_n(12/3)$. This bijection will interchange the $\lb$ and $\rs$ statistics, as well as the $\ls$ and $\rb$ statistics.  Let $w$ be an element of $ R_n(1/23)$. By Theorem~\ref{avoidance}, we know that $w$ is of the form $1^l23\dots m$, with possibly a single one inserted. 
Let $j$ be the number of ones in $w$, and let $i$ be the index of the rightmost one in $w$. We define $\phi: R_n(1/23)\mapsto R_n(12/3)$ as 
$$
\phi(w)=123\dots(n-j+1)(n-i+1)^{j-1}.
$$
From the characterization of $ R_n(12/3)$ provided in Theorem~\ref{avoidance}, it follows that $\phi(w)$ is indeed contained in $ R_n(12/3)$. Furthermore, by Corollary~\ref{avoidancecor} we know that $\# R_n(1/23)=\# R_n(12/3)$. It is also immediate that $\phi$ is injective, which then gives that $\phi$ is a bijection. 

Now we show that $\phi$ takes the $\lb$ statistic to the $\rs$ statistic. First, note that if $w$ is a member of $ R_n(1/23)$ with $\lb(w)=0$, then $w$ must be of the form 
$$
w=1^l23\dots (n-l+1),
$$
for some $l$ with $1\leq l\leq n$. In this case $i=j=l$. Therefore when we apply $\phi$, we are left with 
$$
\phi(w)=123\dots(n-l+1)(n-l+1)^{l-1},
$$
and it follows that $\rs(\phi(w))=0$. Now consider the case where $\lb(w)=k$, for $k>0$. In this instance, $w$ must be of the form 
$$
w=1^l23\dots(k+1)1(k+2)\dots (n-l).
$$ 
It follows that the rightmost one in $w$ has index $l+k+1$, and that there are $l+1$ ones in $w$. Thus when we apply $\phi$, we get 
$$
\phi(w)=123\dots (n-l)(n-l-k)^l,
$$
which satisfies $\rs(\phi(w))=k$. 

Finally, we show that $\phi$ takes the $\ls$ statistic to the $\rb$ statistic.
From the proof of Theorem~\ref{F1/23}, we know that if $w\in  R_n(1/23)$ with maximum value $m$, then $\ls(w)=\binom{m}{2}$. Similarly, from the proof of Theorem~\ref{F12/3}, if $w'\in R_n(12/3)$ with maximum value $m'$, then $\rb(w)=\binom{m'}{2}$. Since $\phi$ preserves maximum values, it follows that $\ls(w)=\rb(\phi(w))$.
\eprf

\section{Open problems and future research}
\label{secopfr}

We have far from exhausted the possible avenues of research concerning these statistics on avoidance classes.  Here are some open problems and indications of future avenues to pursue.

\medskip

1. {\bf Partions of larger sets.}  As we saw in Section~\ref{secmr}, there are interesting results about the pattern $\pi=14/2/3$.  It is natural to consider other partitions of $[n]$ for $n\ge4$.  For example, $\Pi_n(13/24)$ is the set of noncrossing partitions introduced by Kreweras~\cite{kre:pnc}.  We will be considering the noncrossing case in a future paper~\cite{ddggprs:rgf}.

\medskip

2. {\bf Equidistribution.}  In  their original paper, Wachs and White showed that $\lb$ and $\rs$ are equidistributed (have the same generating function) over all RGFs of length $n$ and maximum $m$.  
They also showed that the same holds for $\ls$ and $\rb$.  We have seen that these pairs of statistics are equidistributed over various avoidance classes as well in Corollaries~\ref{1/2/3cor}, \ref{1/23cor}, \ref{13/2cor}, \ref{multicor}, and~\ref{14/2/3cor}.  Is there some more general theorem about equidistribution which will have some (or even all) of these results as special cases?

\medskip

3. {\bf Mahonian pairs.}  It is well known that the permutation statistics $\inv$ and $\maj$ are equidistributed over the symmetric group $\fS_n$.  See Stanley's book~\cite{sta:ec1} for details.  
Any statistic on $\fS_n$ which has this same distribution is said to be {\em Mahonian}.
In~\cite{ss:mp}, Sagan and Savage defined a pair of subsets $(S,T)$ of $\fS_n$ to be a {\em Mahonian pair} if the distribution of $\maj$ over $S$ is the same as the distribution of $\inv$ over $T$.  They give connections of this concept with the Rogers-Ramanujan identities, the Catalan
triangle, and  the Greene-Kleitman decomposition of a Boolean algebra into symmetric chains.  Again, we have seen similar examples in Corollary~\ref{multicor} and Proposition~\ref{mpprop}.  This indicates that exploring the analogous concept for the Wachs and White statistics and avoidance classes should yield interesting results.

\medskip

4.  {\bf RGF avoidance.}  There is a second notion of avoidance for set partitions which we have not touched on in this article.  It is easiest to explain directly in terms of RGFs.  We standardize a sequence of integers by replacing all copies of the smallest element of the sequence by $1$, all copies of the next smallest by $2$, and so on.  Say that an RGF $w$ contains another one $v$ if there is a subsequence of $w$ which standardizes to $v$.  Avoidance is defined in the obvious manner.  If $\si$ avoids $\pi$ then $w(\si)$ avoids $w(\pi)$, but the converse is not always true.  We will be investigating this less restrictive notion of pattern avoidance in a forthcoming article~\cite{ddggprs:rgf}.

\nocite{*}
\bibliographystyle{alpha}

\newcommand{\etalchar}[1]{$^{#1}$}

\end{document}